\documentclass[a4paper,5pt]{article}
\usepackage{graphicx,verbatim,array,multicol,courier}
\usepackage{amsmath,amssymb,amsthm,mathrsfs}
\usepackage[pdftex,bookmarks,colorlinks]{hyperref}
\usepackage{fullpage}
\usepackage{topcapt}
\usepackage{chicago}
\usepackage{color}
\usepackage[section]{placeins} 
\usepackage{mathdots}
\usepackage{mathtools} 
\usepackage{epstopdf} 
\usepackage[labelformat=simple]{subfig}
\usepackage{soul} 
\usepackage{booktabs} 
\usepackage{dsfont}
\usepackage{paralist}
\usepackage[dvipsnames]{xcolor}
\usepackage{enumitem}
\usepackage{setspace}
\definecolor{navy}{rgb}{0,0,0.502}

\hypersetup{colorlinks,%
citecolor=blue,%
filecolor=green,%
linkcolor=red,%
urlcolor=magenta,%
}

\newtheorem{lemma}{Lemma}[section]

\newtheorem{theorem}[lemma]{Theorem}
\newtheorem{coro}[lemma]{Corollary}
\newtheorem{prop}[lemma]{Proposition}
\newtheorem{exam}[lemma]{\normalfont \scshape
 Example}

\newcommand{\R}{\mathbb{R}}
\newcommand{\N}{\mathbb{N}}

\newcommand{\bfG}{\boldsymbol{G}}

\newcommand{\bfY}{\boldsymbol{Y}}
\newcommand{\bfy}{\boldsymbol{y}}

\newcommand{\bflambda}{\boldsymbol{\lambda}}

\DeclareMathOperator{\cov}{Cov}
\DeclareMathOperator{\cor}{Corr}
\DeclareMathOperator{\var}{Var}
\DeclareMathOperator{\expect}{E}
\DeclareMathOperator{\e}{e}

\DeclareMathOperator{\Pro}{Pr}

\def\indic{\mathds{1}}

\def\diff{\,\mbox{d}}

\begin{document}

\title{Some Results on Joint Record Events}%
\author{M. Falk, A. Khorrami Chokami and S. A. Padoan}

\maketitle


\begin{abstract}
Let $X_1,X_2,\dots$ be independent and identically distributed random variables on the real line with a joint continuous distribution function $F$. The stochastic behavior of the sequence of subsequent records is well known. Alternatively to that, we investigate the stochastic behavior of arbitrary $X_j,X_k,j<k$, under the condition that they are records, without knowing their orders in the sequence of records. The results are completely different. In particular it turns out that the distribution of $X_k$, being a record, is not affected by the additional knowledge that $X_j$ is a record as well.
On the contrary, the distribution of $X_j$, being a record, is affected by the additional knowledge that $X_k$ is a record as well. If $F$ has a density, then the gain of this additional information, measured by the corresponding Kullback-Leibler distance, is $j/k$, independent of $F$. We derive the limiting joint distribution of two records, which is not a bivariate extreme value distribution. We extend this result to the case of three records. In a special case we also derive the limiting joint distribution of increments among records.
\end{abstract}

\textbf{Keywords and phrases}: Records, order statistics, Kullback-Leibler distance, domain of attraction, extreme value distribution
\section*{Introduction}

Let $X_1,X_2,\dots$ be independent and identically distributed (iid) random variables (rvs). The rv $X_m$ is a record if $X_m>\max(X_1,\dots,X_{m-1})$.
Clearly, $X_1$ is a record.
Records have been investigated extensively over the past decades, see, e.g. \citeN[Section~4.1]{resn87}, \citeN[Sections~6.2 and 6.3]{gal87}, and \citeN{arnbn98}.
Consider the indicator function
$
I_m:=\indic{(X_m \text{ is a record})},\, m\in\N.
$
It is well known that the indicator functions $I_1,I_2,\ldots$ are independent with (see, e.g., \citeN[Lemma~6.3.3]{gal87})
\begin{equation}\label{eq: p indic}
\Pro(I_m=1)=m^{-1},\quad m\in\N.
\end{equation}
Suppose that the common distribution function (df) $F$ of $X_1,X_2,\dots$ is the standard exponential df $F(x)=1-\exp(-x),\,x\geq 0$.
It is also well known that in this case the increments of subsequent records are iid rvs with
common standard exponential distribution.
Precisely,
put $T(1):=1$ and, for $n\geq 2$,
$
T(n):=\min\{m>n-1\colon X_m \text{ is a record}\}.
$
Then, $X_{T(n)},\,n\in\N$, is the sequence of records among $X_1,X_2,\dots$ and $T(n)$ is the arrival time of the $n$-th record. The increments of subsequent records are given by the sequence
$
Y_n:=X_{T(n)}-X_{T(n-1)},\,n\geq 2,\quad Y_1:=X_1.
$
Then, $Y_1,Y_2,\dots$ are iid rvs with common df $F(x)=1-\exp(-x),\,x\geq 0$. This yields
\begin{equation}\label{eq: df n-th record}
X_{T(n)}=\sum_{i=1}^n Y_i,\,n\in\N,
\end{equation}
and, thus, characterizes the distribution of the $n$-th record or the joint distribution of several numbered records $\left(X_{T(n_1)},X_{T(n_2)},\dots,X_{T(n_m)}\right),n_1<n_2<\dots<n_m$, etc.

In this paper, we drop the assumption that we \emph{know} the order of a record. Therefore, we characterize the distribution
$
\Pro\left(X_j\leq\cdot\mid X_j \text{ is a record}\right),\,j\in\N,
$
as well as the joint distribution of two records
$
\Pro\left(X_j\leq\cdot,X_k\leq\cdot\mid X_j \text{ and $X_k$ are records}\right),\,1\leq j<k.
$
We achieve this under the assumption that the joint df $F$ of $X_1,X_2,\dots$ is continuous.
In particular, we establish
the following surprising fact: Choose integers $j<k$. The distribution of $X_j$, being a record, is affected when we know that $X_k$ is a record as well.
The distribution of $X_k$, being a record, however, is not affected when we know that $X_j$ is a record as well. The corresponding information gain is measured by the Kullback-Leibler distance between the densities. This information gain is $j/k$ and it is independent of the underlying $F$.  This is the content of Section 1.
In Section 2, the asymptotic joint distribution of $X_j$ and $X_k$, suitably standardized, under the condition that they are records, is derived. This is achieved if the underlying df $F$ is in the domain of attraction of an extreme value df. The limit distribution is not an extreme value distribution. We also derive the limiting joint distribution of three records.
Finally, for the special case of a sequence of iid rvs with a common standard negative exponential distribution, we derive the asymptotic joint distribution of increments among records.

\section{Distribution of Records}
Throughout this section we suppose that $X_1,X_2\dots$ are iid rvs with a common continuous df $F$. The distribution of $X_n$, being a record, is provided by the following important result.
\begin{lemma}\label{lemma: df record}
We have for $n\in\N$
\vskip-.6cm
\[
\Pro\left(X_n\leq x\mid X_n\text{ is a record}\,\right)=\Pro\left(\max_{1\leq i\leq n}X_i\leq x\right)=F^n(x),\,x\in\R.
\]
\end{lemma}

Therefore, the distribution of $X_n$, being a record, coincides with that
of the largest order statistic in the sample $X_1,\dots,X_n$.\\
%
{\it Proof.} Denote by $X_{1:n}\leq \dots \leq X_{n:n}$ the order statistics pertaining to $X_1,\dots,X_n$, and by $R(X_i)=\sum_{j=1}^n\indic(X_j\leq X_i)$ the rank of $X_i,\,1\leq i\leq n$.
It is well known that the vector of order statistics $(X_{1:n},\dots,X_{n:n})$ and the vector of ranks $(R(X_1),\dots,R(X_n))$ are independent, with $\Pro(R(X_i)=k)=n^{-1}$, $1\leq i,k\leq n$; see, e.g., \citeN{renyi62}. Therefore, we obtain from equation \eqref{eq: p indic}, the final result
\vskip-1cm
\[
\begin{split}
\Pro\left(X_n\leq x\mid X_n\text{ is a record}\right)&=n\Pro\left(X_n\leq x,X_n\text{ is a record}\right)\\
&=n\Pro\left(X_{n:n}\leq x,R(X_n)=n\right)=\Pro\left(X_{n:n}\leq x\right).\qquad\qquad\qquad\qquad\qquad\Box
\end{split}
\]
%
The preceding result immediately yields the limiting distribution of $X_n$, being a record, as $n$ tends to infinity. The necessary tools are provided by univariate extreme value theory: Suppose that there exist constants $a_n>0,b_n\in\R,n\in\N$, such that
$
F^n\left(a_n x+b_n\right)\to G(x),\, x\in\R,
$
for $n\to\infty$ and for all continuity point $x$ of $G$, where $G$ is a non-degenerate df. Then, $F$ is said to be in the \emph{max-domain of attraction} of $G$, denoted by $F\in\mathcal{D}(G)$, and $G$ is a univariate extreme value distribution.
Precisely, $G$ is a member of a parametric family $\{G_\alpha\colon\alpha\in\R\}$, indexed by $\alpha\in\R$, with
$
G_\alpha(x)=\exp\left({-(1+\alpha x)}^{-1/\alpha}\right),\quad 1+\alpha x>0,
$
if $\alpha$ is different from zero, and the convention
$
G_0(x)=\lim_{\alpha\to 0}G_\alpha(x)=\exp\left(-\e^{-x}\right),\quad x\in\R,
$
(see, e.g., \citeNP{resn87} Ch. 1).
If we put in particular $F(x)=1-\exp(-x),\,x\geq 0$, then we have $F\in \mathcal{D}(G_0)$, precisely
$
F^n(x+\log n)\to\exp\left(-\e^{-x}\right),\, x\in\R,
$
and, thus,
$
\Pro\left( X_n-\log n\leq x\mid X_n\text{ is a record}\right)\to\exp\left(-\e^{-x}\right),\, x\in\R.
$
If we know the order of the records, then the limiting distribution of the $n$-th record is by equation \eqref{eq: df n-th record} and the central limit theorem,
$
{n}^{-1/2}\left(X_{T(n)}-n\right)\overset{d}{\to}\mathcal{N}(0,1), \, n\to\infty,
$
where $\overset{d}{\to}$ denotes  convergence in distribution as $n$ goes to infinity.

Next we establish the joint distribution of two records.
To simplify the notation, we suppose that the underlying df of the sequence of iid rvs
is the standard negative exponential df $F(x)=\exp(x),\,x\leq 0$.
Instead of writing $X_1,X_2,\dots$ we use with this particular underlying df the notation $\eta_1,\eta_2,\dots$.
The latter distribution is a member of the set $\{G_\alpha\colon\alpha\in\R\}$, with $\alpha=-1$ and shifted by $-1$.
In this particular case we have $F^n(x)=\exp(nx)=F(nx)$, $x\le 0$.
\begin{lemma}\label{lemma: bivariate rec exp}
We have for $1\leq j<k$ and $x_1,x_2\in\R$,
\vskip-1cm
\[
\begin{split}
\Pro(\eta_j\le x_1,\eta_k\le x_2\mid \eta_j \text{ and $\eta_k$ are records})
&=\frac{k}{k-j}\Pro\left(\eta_1\leq jx_1,\eta_2\leq(k-j) x_2,(k-j)\eta_1<j\eta_2\right)\\
&=\begin{cases}
\frac{k}{k-j}\left(\e^{(k-j)x_2}\e^{jx_1}-\frac{j}{k}\e^{kx_1}\right), & \text{ if } x_1<x_2\\
\e^{kx_2}=\Pro(\eta_k\leq x_2\mid \eta_k\text{ is a record}), & \text{ if } x_1\geq x_2. \end{cases}
\end{split}
\]
\end{lemma}
\begin{proof}
Let $\eta_1^{(r)},\eta_2^{(r)},\dots$, $r=1,2$, be two independent sequences of iid copies of $\eta$.
Let $\eta_{i:n}^{(r)}$ be $i$-th order statistics and $R_m^{(r)}(\eta_j^{(r)})$ the rank of $\eta_j^{(r)}$ in the sample $\eta_1^{(r)},\dots,\eta_m^{(r)}$.
We split the sample $\eta_1,\dots,\eta_k$ into the two independent sub-samples $(\eta_1,\dots,\eta_j)=:(\eta_1^{(1)},\dots,\eta_j^{(1)})$ and $(\eta_{j+1},\dots,\eta_k)=:(\eta_1^{(2)},\dots,\eta_{k-j}^{(2)})$.
By the independence between vectors of order statistics and ranks and the fact
that the distributions of $\eta_{m:m}$ and $\eta/m$ coincide for $m\in\N$, we obtain
\vskip-1cm
\begin{align*}
&\Pro(\eta_j\le x_1,\eta_k\le x_2\mid \eta_j \text{ and $\eta_k$ are records})\\
&= jk  \Pro\left(\eta_{j:j}^{(1)}\le x_1,\, \eta_{k-j:k-j}^{(2)}\le x_2,\, R_j^{(1)}(\eta_j^{(1)})=j,\, R_{k-j}^{(2)}(\eta_{k-j}^{(2)}) = k-j,\,\eta_{j:j}^{(1)}< \eta_{k-j:k-j}^{(2)} \right)\\
&= jk  \Pro\left(\eta_{j:j}^{(1)}\le x_1,\, \eta_{k-j:k-j}^{(2)}\le x_2,\, \eta_{j:j}^{(1)}< \eta_{k-j:k-j}^{(2)}\right)\Pro\left(R_j^{(1)}(\eta_j^{(1)})=j \right) \Pro\left(R_{k-j}^{(2)}(\eta_{k-j}^{(2)}) = k-j \right)\\
&= \frac k{k-j} \Pro(\eta_1\le jx_1,\,\eta_2\le (k-j)x_2,\,(k-j)\eta_1<j\eta_2).
\end{align*}
The rest of the assertion follows from elementary computations, conditioning on $\eta_2$.
\end{proof}
The preceding result can be extended to $X_1,X_2,\dots$ with an arbitrary continuous df $F$ by putting $X_i:=F^{-1}(\exp(\eta_i)),\,i\in\N$,
where $F^{-1}(q):=\inf\{t\in\R\colon F(t)\geq q\},\,q\in (0,1)$, is the usual generalized inverse of $F$. From the general equivalence $F^{-1}(q)\leq t$ iff $q\leq F(t),\,q\in(0,1),\,t\in\R$, we obtain for $1\leq j<k$ and $y_1,y_2\in\R$
\vskip-1cm
\[
\Pro(X_j\le y_1,X_k\le y_2\mid X_j \text{ and $X_k$ are records})=\Pro(\eta_j\le \log(F(y_1)),\eta_k\le \log(F(y_2))\mid \eta_j \text{ and $\eta_k$ are records}).
\]
By putting $x_i:=\log(F(y_i)),\,i=1,2$, the following result is an immediate consequence of Lemma \ref{lemma: bivariate rec exp}.
\begin{coro}\label{coro: bivariate general X}
We have for integers $1\leq j<k$ and $y_1,y_2\in\R$
\[
\Pro(X_j\le y_1,X_k\le y_2\mid X_j \text{ and $X_k$ are records})
=\begin{cases}
%
\frac{F^j(y_1)}{k-j}(kF^{k-j}(y_2)-jF^{k-j}(y_1)), & \text{if } F(y_1)< F(y_2)\\
F^k(y_2)=\Pro(X_k\leq y_2\mid X_k\text{ is a record}\,), & \text{if } F(y_2)\leq F(y_1). \end{cases}
\]
\end{coro}
Choose integers $1\leq j<k$. Next we establish the fact that the distribution
of $\eta_j$, being a record, is affected, if we know that $\eta_k$ is a record as well. The distribution of $\eta_k$, being a record, however, is not affected by the additional knowledge that $\eta_j$ is a record as well.
\begin{prop}\label{prop: record marginals}
We have for integers $1\leq j<k$ and $x_1,x_2\in\R$,
\vskip-.6cm
\[
\Pro\left(\eta_k\leq x_2\mid\eta_j\text{ and $\eta_k$ are records}\right)=\e^{k x_2}=
\Pro(\eta_k\leq x_2\mid \eta_k\text{ is a record}),\qquad x_2\leq 0,
\]
and
\vskip-.8cm
\[
\begin{split}
\Pro\left(\eta_j\leq x_1\mid\eta_j\text{ and $\eta_k$ are records}\right)
&=\frac{1}{k-j}(k\e^{j x_1}-j\e^{k x_1})\qquad x_1\leq 0,\\
&=\frac{1}{k-j}(k\Pro(\eta_j\leq x_1\mid \eta_j\text{ is a record})-j\Pro(\eta_k\leq x_1\mid \eta_k\text{ is a record})).
\end{split}
\]
\end{prop}
\begin{proof}
From Lemma \ref{lemma: bivariate rec exp} we have
$
\Pro(\eta_k\le x_2\mid \eta_j \text{ and $\eta_k$ are records})
=\frac{k}{k-j}\Pro\left(\eta_2\leq(k-j) x_2,(k-j)\eta_1<j\eta_2\right)
$
by putting $x_1=0$, and
$
\Pro(\eta_j\le x_1\mid \eta_j \text{ and $\eta_k$ are records})
=\frac{k}{k-j}\Pro\left(\eta_1\leq jx_1,(k-j)\eta_1<j\eta_2\right)
$
by putting $x_2=0$. The assertion follows by conditioning on $\eta_2$.
\end{proof}
Let us consider records over a sequence $X_1,X_2,\dots$ of iid rvs with an arbitrary df $F$ that has a density, say $f$. Choose integers $1\leq j<k$. From Corollary \ref{coro: bivariate general X} and Proposition \ref{prop: record marginals} we obtain that the density function of the df
$
G_{j,k}(x):=\Pro(X_j\leq x\mid X_j \text{ and $X_k$ are records})
$
is
$
g_{j,k}(x)=jk{(k-j)}^{-1}f(x)\left(F^{j-1}(x)-F^{k-1}(x)\right),\, x\in\R,
$
and the density function of the df
$
G_{j}(x):=\Pro(X_j\leq x\mid X_j \text{ is a record})
$
is
$
g_{j}(x)=jf(x)F^{j-1}(x),\, x\in\R.
$
%
Suppose $X_j$ is a record.
To summarize by a single number the information, which is inherent in the additional knowledge that $X_k$ is a record as well, we compute the Kullback-Leibler divergence between the density $g_{j,k}$ and the density $g_j$. 
In a general context, the Kullback-Leibler divergence of a density $q(\cdot)$ from a density $p(\cdot)$, is defined by
\vskip-.5cm
\[
D_{KL}(p\| q):=\int_{-\infty}^{+\infty}p(x)\log\frac{p(x)}{q(x)}\diff x.
\]
It quantifies the information lost, when $q(\cdot)$ is used to approximate $p(\cdot)$.
Closely related to the Kullback-Leibler divergence, 
the Kullback-Leibler distance of $p$ and $q$ is defined by
$
D_{KL}(p,q):=D_{KL}(p\| q)+D_{KL}(q\| p).
$
Note that $D_{KL}(p\| q)\geq 0$ by Jensen's inequality.
\begin{prop}\label{prop:KL_distance_X}
The Kullback-Leibler distance between the densities $g_{j,k}$ and $g_j$ is given by
\vskip-.5cm
\[
D_{KL}(g_{j,k} , g_{j})=j/k, \qquad\text{ for $k>j\geq 1$}.
\]
\end{prop}
%
\begin{proof}
Firstly, we show that $D_{KL}(f_{j,k} , f_{j})=j/k$ 
with
$f_{j,k}(x)=jk{(k-j)}^{-1}(\e^{jx}-\e^{kx})$ and $f_j(x)=j\e^{jx}$, $x\le 0$,
i.e. in the case of a sequence of negative exponential random variables.
\[
\begin{split}
D_{KL}(f_{j,k}\| f_j)&=\int_{-\infty}^0\frac{jk}{k-j}\left(\e^{jx}-\e^{kx}\right)\log\left(\frac{jk}{k-j}\left(\e^{jx}-\e^{kx}\right)\frac{\e^{-jx}}{j}\right)\diff x\\
&=\frac{jk}{k-j}\left(\int_{-\infty}^0\left(\e^{jx}-\e^{kx}\right)\log\frac{k}{k-j}\diff x+\int_{-\infty}^0\left(\e^{jx}-\e^{kx}\right)\log\left(1-\e^{(k-j)x}\right)\diff x\right)\\
&=\log\frac{k}{k-j}+\frac{jk}{k-j}\int_{-\infty}^0\left(\e^{jx}-\e^{kx}\right)\log\left(1-\e^{(k-j)x}\right)\diff x.
\end{split}
\]
The substitution $t=1-\e^{(k-j)x}$ entails
\[
\int_{-\infty}^0\left(\e^{jx}-\e^{kx}\right)\log\left(1-\e^{(k-j)x}\right)\diff x=\frac{1}{k-j}\int_0^1\left({\left(1-t\right)}^{\frac{j}{k-j}-1}-{\left(1-t\right)}^{\frac{k}{k-j}-1}\right)\log t\diff t.
\]
Note that 
\[
\int_0^1{\left(1-t\right)}^{\frac{j}{k-j}-1}\log t\diff t=B\left(1,\frac{j}{k-j}\right)\left(\psi(1)-\psi\left(1+\frac{j}{k-j}\right)\right)=\frac{k-j}{j}\left(\psi(1)-\psi\left(1+\frac{j}{k-j}\right)\right)
\]
%
where $B(a,b)=\int_0^1 t^{a-1}{(1-t)}^{b-1}\diff t,\,a,b>0$, $\psi(x)=\Gamma'(x)/\Gamma(x),\,x>0$, $\Gamma(x)=\int_0^\infty t^{x-1}\e^{-t}\diff t,\,x>0$ are the Beta, Digramma and Gamma functions, respectively.
Analogously, one obtains
\[
\int_0^1{\left(1-t\right)}^{\frac{k}{k-j}-1}\log t\diff t=\frac{k-j}{k}\left(\psi(1)-\psi\left(1+\frac{k}{k-j}\right)\right).
\]
As a consequence, we obtain
\[
D_{KL}(f_{j,k}\| f_j)=\log\frac{k}{k-j}+\psi(1)+\frac{j}{k-j}\psi\left(1+\frac{k}{k-j}\right)-\frac{k}{k-j}\psi\left(1+\frac{j}{k-j}\right).
\]
Furthermore, we have,
\[
\begin{split}
D_{KL}(f_j\| f_{j,k})&=-\int_{-\infty}^0 j\e^{jx}\log\left(\frac{jk}{k-j}\left(\e^{jx}-\e^{kx}\right)\frac{\e^{-jx}}{j}\right)\diff x\\
&=-\log\frac{k}{k-j}\int_{-\infty}^0 j\e^{jx}\diff x-j\int_{-\infty}^0\e^{jx}\log\left(1-\e^{(k-j)x}\right)\diff x\\
&=-\log\frac{k}{k-j}-\frac{j}{k-j}\int_0^1{\left(1-t\right)}^{\frac{j}{k-j}-1}\log t\diff t=-\log\frac{k}{k-j}-\psi(1)+\psi\left(1+\frac{j}{k-j}\right).
\end{split}
\]
Finally, we obtain
\[
D_{KL}(f_{j,k}\| f_j)+D_{KL}(f_j\| f_{j,k})=\frac{j}{k-j}\left(\psi\left(1+\frac{k}{k-j}\right)-\psi\left(1+\frac{j}{k-j}\right)\right).
\]
The functional equation $\psi(1+x)=\psi(x)+1/x$, $x>0$, implies
\[
\psi\left(1+\frac{k}{k-j}\right)-\psi\left(1+\frac{j}{k-j}\right)=\psi\left(1+\frac{k}{k-j}\right)-\psi\left(\frac{k}{k-j}\right)=\frac{k-j}{k},
\]
which yields the assertion. For the case of a general sequence of random variables,
we have
\[
D_{KL}(g_{j,k}\| g_j)=\int_{-\infty}^{+\infty}\frac{jk}{k-j}f(x)\left(F^{j-1}(x)-F^{k-1}(x)\right)\log\left(\frac{k}{k-j}\frac{F^{j-1}(x)-F^{k-1}(x)}{F^{j-1}(x)}\right)\diff x.
\]
The substitution $t=F^{-1}\left(\exp(x)\right)$ entails that the above integral equals
\[
\int_{-\infty}^0\frac{jk}{k-j}\left(\e^{jx}-\e^{kx}\right)\log\left(\frac{jk}{k-j}\left(\e^{jx}-\e^{kx}\right)\frac{\e^{-jx}}{j}\right)\diff x=D_{KL}(f_{j,k}\| f_j).
\]
Equally, one shows that
$
D_{KL}(g_{j}\| g_{j,k})=D_{KL}(f_{j}\| f_{j,k}).
$
\end{proof}
%
%
%
%
%
Clearly, $0<D_{KL}(g_{j,k}, g_j)<1$. The Kullback-Leibler distance between the densities $g_{j,k}$ and $g_j$ gets small if $j/k$ gets small. This means that the additional knowledge that $X_k$ is a record as well, affects the distribution of $X_j$, being a record, less if $k$ gets large.
%
On the other hand, if $k=j+1$, then the information gain approaches one if $j$ gets large.
%
%
%
%
By repeating the arguments in the proof of Lemma \ref{lemma: bivariate rec exp}, one derives the joint distribution of an arbitrary number of records as it is established by the next result.
\begin{lemma}\label{lemma: auxiliary}
We have for integers $1\le j_1\dots< j_{d},\,d\in\N$, with $j_0=0$, and $x_1,\dots,x_d\leq 0$,
\vskip-1cm
\begin{align*}
&\Pro(\eta_{j_m}\le x_m, 1\leq m\leq d\mid \eta_{j_1},\dots,\eta_{j_d}\text{ are records})\\
&= \frac{\prod_{m=2}^{d}j_m}{\prod_{m=2}^{d}(j_{m}-j_{m-1})}\Pro\left(\frac{\eta_m}{j_m-j_{m-1}}\le x_m, \frac{(j_{m+1}-j_m)\eta_m}{j_{m}-j_{m-1}}
< \eta_{m+1}, 1\leq m\leq d-1,\frac{\eta_d}{j_d-j_{d-1}}\leq x_d
\right).
\end{align*}
\end{lemma}
The case of an arbitrary sequence of iid rvs $X_1,X_2,\dots$ with common continuous df $F$ can immediately be deduced from the preceding result via the representation $X_i=F^{-1}(\exp(\eta_i))$, $i\in\N$.
\section{Asymptotic Joint Distribution of Records}
Let $X_1,X_2,\dots$ be iid rvs with common df $F$, which is in the domain of attraction of an extreme value distribution $G$.
From Lemma \ref{lemma: df record} we immediately obtain the following result.
\begin{lemma}
Under the preceding conditions we obtain
\[
\Pro\left(\frac{X_n-b_n}{a_n} \leq x\mid X_n \text{ is a record}\right)\xrightarrow[n\to\infty]{}G(x),\qquad x\in\R.
\]
\end{lemma}
In what follows we investigate the joint asymptotic distribution of two records. We start with a sequence $\eta_1,\eta_2\dots$ of iid rvs that follow the standard negative exponential df $F(x)=\exp(x),\,x\leq 0$. From Lemma \ref{lemma: bivariate rec exp} we obtain for $x_1,x_2\in\R$,
\[
\Pro\left(\eta_j\le \frac{x_1}{j},\eta_k\le \frac{x_2}{k}\mid \eta_j \text{ and $\eta_k$ are records}\right)
=\frac{k}{k-j}\Pro\left(\eta_1\leq x_1,\eta_2\leq\frac{k-j}{k} x_2,\frac{k-j}{j}\eta_1<j\eta_2\right).
\]
We let $j=j(n)$ and $k=k(n)$ both depend on $n\in\N$ with
\begin{align}\label{eq: 1cond}
\lim_{n\to\infty}\frac jn &=\lambda_1 >0,\quad \lim_{n\to\infty}\frac kn =\lambda_2>\lambda_1.
\end{align}
The next result is a consequence of Lemma \ref{lemma: bivariate rec exp} and elementary computations.
\begin{prop}\label{prop: lim joint}
Under condition \eqref{eq: 1cond}, for all $x_1,x_2\leq 0,\,\beta_j=\lambda_j/(\lambda_2-\lambda_1)$ and $j=1,2$, we obtain
\[
\lim_{n\to\infty}\Pro\left(\eta_j\le \frac{x_1}{n},\eta_k\le \frac{x_2}{n}\mid \eta_j \text{ and $\eta_k$ are records}\right)=H_{\lambda_1,\lambda_2}(x_1,x_2),
\]
where
\begin{equation}\label{eq:cdf_uni_reco}
H_{\lambda_1,\lambda_2}(x_1,x_2)=\begin{cases}
 \e^{\lambda_1 x_1}(\beta_2\,\e^{(\lambda_2-\lambda_1)x_2} - \beta_1\,\e^{(\lambda_2-\lambda_1)x_1}),& \text{ if } x_1<x_2,\\
 \e^{\lambda_2 x_2}, &\text{ if }  x_1\geq x_2.
\end{cases}
\end{equation}
\end{prop}
The  marginal df of $H_{\lambda_1,\lambda_2}$ are
\begin{equation}\label{eq:marginal_X_Y}
H_{1}(x)= H_{\lambda_1,\lambda_2}(x,0)=\beta_2\,\e^{\lambda_1 x} - \beta_1\,\e^{\lambda_2 x},\quad 
H_{2}(x)= H_{\lambda_1,\lambda_2}(0,x)=\e^{\lambda_2 x},\quad x\leq 0.
\end{equation}
Clearly, the fact that $H_2$ is independent of $\lambda_1$ reflects the fact that the distribution of $\eta_k$, being a record, is not affected by the additional knowledge that $\eta_j$ is a record as well, as shown in the previous section.
While $H_2$ is a univariate extreme value distribution, $H_1$ is not. Therefore, the bivariate df $H_{\lambda_1,\lambda_2}$ is not a multivariate extreme value distribution.
In the next result we provide the marginal means, variances and the covariance of the margins of $H_{\lambda_1,\lambda_2}$.

\begin{prop}\label{prop: uni_dep_measures}
Let $(X,Y)$ be a bivariate rv with df given in \eqref{eq:cdf_uni_reco}.
Then, we have for all $\lambda_2>\lambda_1>0$
\begin{enumerate}

\item $
\expect(X)=-\lambda_1^{-1}-\lambda_2^{-1},\quad \var(X)=\lambda_1^{-2}+\lambda_2^{-2},\quad \expect(Y)=-\lambda_2^{-1}, \quad
\var(Y)=\lambda_2^{-2}
$

%
%
\item $
\cov(X,Y)=\lambda_2^{-2},\quad
\cor(X,Y)=\frac{\lambda_1}{\sqrt{\lambda_1^2+\lambda_2^2}},\quad
\expect\left[(X-Y)^2\right] = \lambda_1^{-2}.
$
\end{enumerate}
%
%
\end{prop}
\begin{proof}
Assume that the probability law of the pairs of the rvs $(X,Y)$ is given by \eqref{eq:cdf_uni_reco}.
Then
%
\[
\begin{split}
\cov(X,Y)&=\int_{-\infty}^0\int_{-\infty}^0 H_{{\lambda_1,\lambda_2}}(x,y)-H_{{\lambda_1,\lambda_2}}(x,0)H_{{\lambda_1,\lambda_2}}(0,y) \diff x \diff y\\
&=\int_{-\infty}^0
\left(
\int_{-\infty}^y
\frac{\lambda_2\,\e^{(\lambda_2-\lambda_1)y+\lambda_1x} - \lambda_1\,\e^{\lambda_2 x}}
{\lambda_2-\lambda_1} \diff x +
\int_{y}^0 \e^{\lambda_2 y} \diff x
\right)
\diff y-\int_{-\infty}^0 \frac{\lambda_2\,\e^{\lambda_1 x} - \lambda_1\,\e^{\lambda_2 x}}
{\lambda_2-\lambda_1} \diff x
\int_{-\infty}^0 \e^{\lambda_2y}\diff y\\
&= \frac{\lambda_2^2+\lambda_1\lambda_2-2\lambda_1^2}{\lambda_1\lambda_2^2(\lambda_2-\lambda_2)}-\frac{\lambda_1+\lambda_2}{\lambda_1\lambda_2}
\frac{1}{\lambda_2}
=\frac{1}{\lambda_2^2}.
\end{split}
\]
%
The variance of $X$ is
\[
\var(X)=\expect\left(X^2\right)-\expect^2(X)=2\int_{-\infty}^0-x\,\frac{\lambda_2\,\e^{\lambda_1 x} - \lambda_1\,\e^{\lambda_2 x}}
{\lambda_2-\lambda_1}\diff x-\left(-\int_{-\infty}^0\frac{\lambda_2\,\e^{\lambda_1 x} - \lambda_1\,\e^{\lambda_2 x}}
{\lambda_2-\lambda_1}\diff x\right)^2
=\frac{\lambda_1^2+\lambda_2^2}{\lambda_1^2\lambda_2^2}.
\]
The marginal distribution of $Y$ is $\exp(\lambda_2y)$, $y\leq 0$, therefore its
mean and variance are $1/\lambda_2$ and $1/\lambda_2^2$, respectively. Finally, combining these results
\[
\cor(X,Y)=\frac{\cov(X,Y)}{\sqrt{\var(X)\var(Y)}}=
\frac{1/\lambda_2^2}{	\sqrt{(\lambda_1^2+\lambda_2^2)/\lambda_1^2\lambda_2^2 \cdot 1/\lambda_2^2}}
=\frac{\lambda_1}{\sqrt{\lambda_1^2+\lambda_2^2}}.\qedhere
\]
\end{proof}
%
%
%
The next result extends Proposition \ref{prop: lim joint} to a sequence of iid rvs, whose df $F$ satisfies $F\in\mathcal{D}(G)$.
\begin{coro}\label{coro: joint_rec_GEV}
Let $X_1,X_2,\dots$ be iid copies of a rv $X$ with a continuous distribution $F$. Assume that $F\in\mathcal{D}(G)$ with norming constants $a_n>0$ and $b_n\in\R$, $n\in\N$. Then, under Condition \eqref{eq: 1cond}, we have for $y_1,y_2\in\R$
\[
\lim_{n\to\infty}\Pro\left(\frac{X_j-b_n}{a_n}\leq y_1, \frac{X_k-b_n}{a_n}\leq y_2 \mid X_j \text{ and $X_k$ are records}\right)=G_{\lambda_1,\lambda_2}(y_1,y_2),
\]
where
\begin{equation*}
G_{\lambda_1,\lambda_2}(y_1,y_2) = \begin{cases}
G^{\lambda_1}(y_1)\left(\beta_2G(y_2)^{\lambda_2-\lambda_1}-\beta_1G(y_1)^{\lambda_2-\lambda_1}\right),&\mbox{ if }y_1<y_2,\\
G^{\lambda_2}(y_2),&\mbox{ if } y_1\ge y_2.
\end{cases}
\end{equation*}
\end{coro}
The marginal distributions are given by
$
G_{\lambda_1,\lambda_2}(y_1,\infty) = \beta_2 G(y_1)^{\lambda_1}-\beta_1 G(y_1)^{\lambda_2},\,
G_{\lambda_1,\lambda_2}(\infty,y_2) = G(y_2)^{\lambda_2},\, y_1, y_2\in\R;
$
note that the second marginal is independent of $\lambda_1$.
Note that results on the limiting distribution of joint records with \emph{known} orders in the sequence of records have been recently derived by \citeN{barakat2017asymptotic}.
\begin{proof}
Put $\eta_m:=\log(F(X_m))$, $m\in\N$. Then $\eta_1,\eta_2,\dots$ are iid standard negative exponential distribution. Since $\log(\cdot)$ and $F(\cdot)$ are monotone we obtain
\begin{align*}
&\Pro\left(X_j\leq a_ny_1+b_n,X_k\leq a_ny_2+b_n\mid X_j \text{ and $X_k$ are records}\right)\\
&= \Pro\left(\eta_j \leq \frac{n\log(F(a_ny_1+b_n))}n, \eta_k \leq \frac{n\log(F(a_ny_2+b_n))}n\mid \eta_j \text{ and $\eta_k$ are records}\right).
\end{align*}
The condition
$
F^n(a_nx+b_n)\to G(x),\, x\in\R,
$
as $n\to\infty$, is equivalent to
$
n\log(F(a_nx+b_n))\to\log(G(x)),\, 0<G(x)\le 1,
$
as $n\to\infty$.
Proposition \ref{prop: lim joint} now implies
\[
\begin{split}
&\Pro\left(\eta_j \leq \frac{n\log(F(a_ny_1+b_n))}n, \eta_k \leq \frac{n\log(F(a_ny_2+b_n))}n\mid \eta_j, \text{$\eta_k$ are records}\right)\\
&\xrightarrow[n\to\infty]{}H_{\lambda_1,\lambda_2}(\log(G(y_1)),\log(G(y_2))).\qedhere
\end{split}
\]
\end{proof}
We have established the fact that the distribution of $X_k$, being a record, is not affected if we know in addition that $X_j$ is a record as well. But what happens if we know, for example, that $X_j$, being a record, has already exceeded a fixed threshold? The answer is a straightforward consequence of our preceding results. We obtain for $y>u\in\R$, under the conditions of Corollary \ref{coro: joint_rec_GEV},
\[
\begin{split}
&\Pro\left( \frac{X_k-b_n}{a_n} \leq y\, \mid \frac{X_j-b_n}{a_n}>u, X_j \text{ and $X_k$ are records}\right)\\
&=\frac{\Pro\left( \frac{X_k-b_n}{a_n} \leq y\mid X_j \text{ and $X_k$ are records}\right)-\Pro\left(\frac{X_j-b_n}{a_n}\leq u,\frac{X_k-b_n}{a_n} \leq y\mid X_j \text{ and $X_k$ are records}\right)}{1-\Pro\left( \frac{X_j-b_n}{a_n}\leq u\mid X_j \text{ and $X_k$ are records}\right)}\\
&\xrightarrow[n\to\infty]{}\frac{G(y)^{\lambda_2}-G^{\lambda_1}(u)\left(\beta_2G(y)^{\lambda_2-\lambda_1}-\beta_1G(u)^{\lambda_2-\lambda_1}\right)}{\beta_2\left(1-G(u)^{\lambda_1}\right) -\beta_1\left(1-G(u)^{\lambda_2}\right)}.
\end{split}
\]
The results obtained so far can be extended to the case of an arbitrary number of records. However, computations become really hard. We report the case of the asymptotic joint df of three records.
\begin{prop}\label{prop: cdf trivariate}
Let $X_1,X_2,\dots$ be iid copies of a rv $X$ with a continuous distribution $F$. Assume that $F\in\mathcal{D}(G)$ with norming constants $a_n>0$ and $b_n\in\R$, $n\in\N$.
Assume also that $j=j(n)$, $k=k(n)$ and $r=r(n)$ all depending on $n\in\N$ with $j<k<r$ and
\begin{align*}
\lim_{n\to\infty}\frac jn &=\lambda_1 >0,\quad \lim_{n\to\infty}\frac kn =\lambda_2>\lambda_1,
\quad \lim_{n\to\infty}\frac r n =\lambda_3>\lambda_2.
\end{align*}
Then, for all $\bfy\in\R^3$, we have
\[
\Pro\left(\frac{X_j-b_n}{a_n}\leq y_1, \frac{X_k-b_n}{a_n}\leq y_2, \frac{X_r-b_n}{a_n}\leq y_3\mid X_j,\,X_k\text{ and $X_r$ are records} \right)\to \bfG_{\bflambda}(\bfy),
\]
as $n\to\infty$, where
\begin{equation*}\label{eq: cdf trivariate}
\bfG_{\bflambda}(\bfy) =
\begin{cases}
G(y_1)^{\lambda_1}G(y_2)^{\lambda_2-\lambda_1}\left(\beta_2\beta_6 G(y_3)^{\lambda_3-\lambda_1}-\beta_4\beta_5 G(y_2)^{\lambda_3-\lambda_1}\right)\\
 \quad -G(y_1)^{\lambda_2}
\left(\beta_1\beta_6 G(y_3)^{\lambda_3-\lambda_2}-\beta_3\beta_4 G(y_1)^{\lambda_3-\lambda_2}\right)\quad\text{if}\quad y_1\leq y_2\leq y_3\\
G(y_2)^{\lambda_2}\left(\beta_6 G(y_3)^{\lambda_3-\lambda_2} - \beta_4 G(y_2)^{\lambda_3-\lambda_2}\right),
\quad \text{if}\quad  y_2\leq y_1\leq y_3 \text{ or } y_2\leq y_3\leq y_1,\\
\beta_2\beta_5 G(y_1)^{\lambda_1} G(y_3)^{\lambda_3-\lambda_1}  - \beta_1\beta_6G(y_1)^{\lambda_2}
G(y_3)^{\lambda_3-\lambda_2} + \beta_3\beta_4 G(y_1)^{\lambda_3},
\quad \text{if}\quad   y_1\leq y_3\leq y_2,\\
G(y_3)^{\lambda_3},\quad\text{if}\quad y_3\leq y_2\leq y_1\text{ or } y_3\leq y_1\leq y_2,
\end{cases}
\end{equation*}
and where $\beta_1, \beta_2$ are as in Proposition \ref{prop: lim joint}, $\beta_3=\lambda_1/(\lambda_{3}-\lambda_1)$, $\beta_4=\lambda_2/(\lambda_{3}-\lambda_2)$,
$\beta_5=\lambda_3/(\lambda_{3}-\lambda_1)$, $\beta_6=\lambda_3/(\lambda_{3}-\lambda_2)$, and $\bflambda=(\lambda_1,\lambda_2,\lambda_3)$.
In addition, let  $\bfY=(Y_1,Y_2,Y_3)$ be a rv with df $\bfG_{\bflambda}(\bfy)$, then
the variance-covariance matrix of $\bfY$ is
\begin{equation*}
\left(
\begin{array}{ccc}
\lambda_1^{-2}+\lambda_2^{-2}+\lambda_3^{-2} &
\lambda_2^{-2}+\lambda_3^{-2}
& \lambda_3^{-2}\\
  & \lambda_2^{-2}+\lambda_3^{-2} &
  \lambda_3^{-2}\\
 &  & \lambda_3^{-2}\\
\end{array}
\right).
\end{equation*}
\end{prop}
\begin{proof}
Let $\eta_1,\eta_2,\ldots$ be iid rvs with a common negative exponential distribution.
%
First of all we compute the following non-asymptotic distribution when $x_1<x_2<x_3$
\vskip-0.5cm
\[
\begin{split}
&\Pro\left(\eta_j\leq x_1,\eta_k\leq x_2,\eta_r\leq x_3\mid \eta_j,\,\eta_k\text{ and $\eta_r$ are records}\right)\\
&=\frac{kr}{(k-j)(r-k)}\Pro\left(\eta_1\leq jx_1,\eta_2\leq(k-j) x_2,\eta_3\leq(r-k)x_3,(k-j)\eta_1<j\eta_2,(r-k)\eta_2<(k-j)\eta_3\right)\\
&=\frac{kr}{(k-j)(r-k)}\int_{-\infty}^{jx_1}\int_{\frac{k-j}{j}z_1}^{(k-j)x_2}\Pro\left(\frac{r-k}{k-j}z_2<\eta_3\leq x_3\right)\e^{z_2+z_1}\diff z_2\diff z_1\\
&=\frac{kr}{(k-j)(r-k)}\int_{-\infty}^{jx_1}\e^{z_1}\left(\e^{(r-k)x_3}\left(\e^{(k-j)x_2}-\e^{\frac{k-j}{j}z_1}\right)-\frac{k-j}{r-j}\left(\e^{(r-j)x_2}-\e^{\frac{r-j}{j}z_1}\right)\right)\diff z_1\\
&=\frac{kr}{(k-j)(r-k)}\left(\e^{jx_1}\left(\e^{(r-k)x_3}\e^{(k-j)x_2}-\frac{k-j}{r-j}\e^{(r-j)x_2}\right)-\frac{j}{k}\e^{kx_1}\e^{(r-k)x_3}+\frac{j(k-j)}{r(r-j)}\e^{rx_1}\right).
\end{split}
\]
The cases of $x_2< x_1< x_3$ is obtained from the
expression of the above formula by substituting $x_2$ in $x_1$. Similarly the case
$x_1< x_3< x_2$ is obtained by substituting $x_3$ in $x_2$ and lastly the case
$x_3< x_2< x_1$ is obtained by substituting
$x_3$ in both $x_1$ and $x_2$.
Then, the asymptotic distribution is easily obtained by computing
$
\lim_{n\to\infty} \Pro\left(\eta_j\leq x_1/n,\eta_k\leq x_2/n,\eta_r\leq x_3/n \mid
\eta_j,\,\eta_k \text{ and $\eta_r$ are records}\right).
$
The case of an arbitrary distribution can be deduced by following the same reasoning of Corollary \ref{coro: joint_rec_GEV} and therefore the first assertion is derived.
We compute the variance-covariance matrix. Note that the bivariate and univariate marginal distribution functions of $(Y_1,Y_2)$ are
%
%

%
\vskip-1cm
\begin{align*}
&F_{\lambda_1,\lambda_2}(x_1,x_2)=
\begin{cases}
&\left(\frac{\lambda_2\lambda_3\e^{(\lambda_2-\lambda_1)x_2}}{(\lambda_2-\lambda_1)(\lambda_3-\lambda_2)}-\frac{\lambda_2\lambda_3\e^{(\lambda_3-\lambda_1)x_2}}{(\lambda_3-\lambda_2)(\lambda_3-\lambda_1)}\right)\e^{\lambda_1x_1}-\frac{\lambda_1\lambda_3\e^{\lambda_2x_1}}{(\lambda_2-\lambda_1)(\lambda_3-\lambda_2)}-\frac{\lambda_1\lambda_2\e^{\lambda_3 x_1}}{(\lambda_3-\lambda_1)(\lambda_3-\lambda_2)},\quad\text{if}\quad x_1\leq x_2,\\
& \frac{\lambda_3\e^{\lambda_2x_2} - \lambda_2\e^{\lambda_3x_2}}{\lambda_3-\lambda_2},
\quad \text{if}\quad  x_2\leq x_1,\\
\end{cases}\\
&F_{\lambda_1}(x_1)=\frac{\lambda_2\lambda_3}{(\lambda_2-\lambda_1)(\lambda_3-\lambda_1)}\e^{\lambda_1x_1}-\frac{\lambda_1\lambda_3}{(\lambda_2-\lambda_1)(\lambda_3-\lambda_2)}\e^{\lambda_2x_1}-\frac{\lambda_1\lambda_2}{(\lambda_3-\lambda_1)(\lambda_3-\lambda_2)}\e^{\lambda_3 x_1},\,x_1\leq 0,\\
&F_{\lambda_2}(x_2)= \frac{\lambda_3\e^{\lambda_2 x_2} - \lambda_2\e^{\lambda_3 x_2}}{\lambda_3-\lambda_2},\,x_2\leq 0,
\end{align*}
where we have taken the transformation $x_j=\log(G(y_j))$, $j=1,\ldots,3$, for simplicity.
Hoeffding's covariance identity implies that
\[
\begin{split}
\cov(Y_1,Y_2)&=\int_{{(-\infty,0]}^2}F_{\lambda_1,\lambda_2}(x_1,x_2)\diff x_1\diff x_2-\int_{-\infty}^0F_{\lambda_1}(x_1)\diff x_1\int_{-\infty}^0F_{\lambda_2}(x_2)\diff x_2\\
&=\frac{\lambda_3}{\lambda_1(\lambda_2-\lambda_1)(\lambda_3-\lambda_2)}-\frac{\lambda_2}{\lambda_1(\lambda_3-\lambda_1)(\lambda_3-\lambda_2)}-\frac{\lambda_1\lambda_3}{\lambda_2^2(\lambda_2-\lambda_1)(\lambda_3-\lambda_2)}\\
&\quad+\frac{\lambda_1\lambda_2}{\lambda_3^2(\lambda_3-\lambda_1)(\lambda_3-\lambda_2)}+\frac{\lambda_3}{\lambda_2^2(\lambda_3-\lambda_2)}-\frac{\lambda_2}{\lambda_3^2(\lambda_3-\lambda_2)}-\left(\frac{1}{\lambda_1}+\frac{1}{\lambda_2}+\frac{1}{\lambda_3}\right)\left(\frac{1}{\lambda_2}+\frac{1}{\lambda_3}\right).
\end{split}
\]
By straightforward simplifications we obtain $\cov(Y_1,Y_2)=\lambda_2^{-2}+\lambda_3^{-2}$.
The other covariances are computed in a similar way. 
The expected values and the variances are derived by simple computations.
\end{proof}
Under Condition \eqref{eq: 1cond},
another application of Lemma \ref{lemma: auxiliary} yields the following results.
%
%
\begin{theorem}\label{teo: joint rec-diff}
Let $\eta_1,\eta_2,\dots$ be independent and standard negative exponential distributed rvs.
Assume that $j_i=j_i(n)\in\N$, $i=1,2,\ldots$, $n=2,3,\dots$ are sequences of integers satisfying
$
\lim_{n\to\infty}j_i/n =\lambda_i >0,
$
with $0<\lambda_1 < \lambda_2 < \cdots$.
Under these conditions and every $x\leq 0$, $y,y_1,\ldots,y_s>0$ and $m\in\N$, we have
\[
\begin{aligned}
%
&\lim_{n\to\infty}\Pro\left(\eta_{j_{i+1}}-\eta_{j_i}\leq y_i/n,1\leq i\leq s\mid \eta_{j_1}\dots\eta_{j_s} \text{ are records}\right)  = \prod_{i=1}^s\left(1-\e^{-\lambda_i y_i}\right),\\
&\lim_{n\to\infty}\Pro\left(\eta_j\leq x/n,\eta_k-\eta_j\leq y/n\mid \eta_j \text{ and $\eta_k$ are records}\right) = Q_{{\lambda_1,\lambda_2}}(x,y),\\
\end{aligned}
\]
where
\begin{equation}\label{eq:cdf_int}
Q_{{\lambda_1,\lambda_2}}(x,y)=
\begin{cases}
\beta_1 \left(\e^{(\lambda_2-\lambda_1)y}-1\right)\e^{\lambda_2 x}, &\text{ if }|x|\geq\,y,\\
\beta_2 \e^{\lambda_1 x}-\beta_1\e^{\lambda_2 x}-\e^{-\lambda_1 y}, &\text{ if }|x|<\,y.
\end{cases}\\
\end{equation}
\end{theorem}
The marginal distributions of \eqref{eq:cdf_int} are 
$
Q_{{\lambda_1,\lambda_2}}(x)= Q_{{\lambda_1,\lambda_2}}(x,\infty)=\beta_2\,\e^{\lambda_1 x} - \beta_1\,\e^{\lambda_2 x},\, x\leq 0,
$
and
$
Q_{\lambda_1}(y)= Q_{{\lambda_1,\lambda_2}}(0,y)=1 - \e^{-\lambda_1 y},\, y > 0.
$
These results mean that the increments $Y_1,\ldots,Y_n$ among records are independent but not identically distributed. Furthermore, a generic record $\eta_j$, $j=1,2,\ldots$ and
the increment between two records $\eta_j$ and $\eta_k,\,k>j$ are not independent.
\begin{proof}
For all $y_1,\ldots,y_s >0$, by following the same reasoning in the proof of Lemma \ref{lemma: bivariate rec exp} we have
\begin{align*}
&\Pro\left(\eta_{j_2}-\eta_{j_{1}} > y_1,\ldots, \eta_{j_s}-\eta_{j_{s-1}} > y_{s-1}\mid \eta_{j_1}\dots\eta_{j_s} \text{ are records}\right)\\
&= \frac{\prod_{m=2}^{s}j_m}{\prod_{m=2}^{s}(j_{m}-j_{m-1})}
\Pro\left( \frac{\eta_2}{j_2-j_1}-\frac{\eta_1}{j_1} > z_1,\ldots,
\frac{\eta_s}{j_{s}-j_{s-1}}-\frac{\eta_{s-1}}{j_{s-1}} > z_{s-1}\right)=\frac{\prod_{m=2}^{s}j_m}{\prod_{m=2}^{s}(j_{m}-j_{m-1})}\cdot A
\end{align*}
where
$$
A=\int_{-\infty}^0\int_{-\infty}^{(j_{s-1}-j_{s-2})\left(\frac{z_s}{j_s-j_{s-1}}-y_{s-1}\right)}\cdots\int_{-\infty}^{(j_2-j_1)\left(\frac{z_3}{j_3-j_2}-y_2\right)}
\Pro\left(\eta_1<j_1\left(\frac{z_2}{j_2-j_1}-y_1\right)\right) \prod_{i=2}^s\e^{z_i}\diff z_i.
$$
We show by induction that
\[
\int_{-\infty}^{(j_{m}-j_{m-1})\left(\frac{z_{m+1}}{j_{m+1}-j_m}-y_m\right)}\cdots\int_{-\infty}^{(j_2-j_1)\left(\frac{z_3}{j_3-j_2}-y_2\right)}
\e^{\frac{j_2}{j_2-j_1}z_2}\e^{-j_1y_1}\prod_{i=2}^m\e^{z_i}\diff z_i=\frac{\prod_{i=2}^m(j_i-j_{i-1})}{\prod_{i=2}^mj_i}\prod_{i=1}^{m}\e^{-j_iy_i}\e^{\frac{j_m z_{m+1}}{j_{m+1}-j_m}}.
\]
At the step 1 we have
$$
\int_{-\infty}^{(j_2-j_1)\left(\frac{z_3}{j_3-j_2}-y_2\right)}
\e^{\frac{j_2}{j_2-j_1}z_2}\e^{-j_1y_1}\diff z_2=\frac{j_2-j_1}{j_2}\e^{-j_1y_1-j_2y_2}\e^{\frac{j_2 z_3}{j_3-j_2}}.
$$
True for $m$. At the step $m+1$ we have
\[
\begin{split}
&\int_{-\infty}^{(j_{m+1}-j_m)\left(\frac{z_{m+2}}{j_{m+2}-j_{m+1}}-y_{m+1}\right)}\cdots\int_{-\infty}^{(j_2-j_1)\left(\frac{z_3}{j_3-j_2}-y_2\right)}
\e^{\frac{j_2}{j_2-j_1}z_2}\e^{-j_1y_1}\prod_{i=2}^{m+1}\e^{z_i}\diff z_i\\
&=\int_{-\infty}^{(j_{m+1}-j_m)\left(\frac{z_{m+2}}{j_{m+2}-j_{m+1}}-y_{m+1}\right)}\frac{\prod_{i=2}^m(j_i-j_{i-1})}{\prod_{i=2}^mj_i}\prod_{i=1}^{m}\e^{-j_iy_i}\e^{\frac{j_m z_{m+1}}{j_{m+1}-j_m}}\e^{z_{m+1}}\diff z_{m+1}\\
&=\frac{\prod_{i=2}^{m+1}(j_i-j_{i-1})}{\prod_{i=2}^{m+1}j_i}\prod_{i=1}^{m+1}\e^{-j_iy_i}\e^{\frac{j_{m+1} z_{m+2}}{j_{m+2}-j_{m+1}}}.
\end{split}
\]
As a consequence
$
A=\prod_{i=2}^{s}(j_i-j_{i-1}){\prod_{i=2}^{s}j_i}^{-1}\prod_{i=1}^{s-1}\e^{-j_iy_i}.
$
We obtain the first result
which shows that the increments among records are independent exponentials but with different parameters $0<\lambda_1<\lambda_2<\cdots<\lambda_s$.
By the assumptions we have
that $(1 - \e^{j_{i}y_{i}/n})\to (1 - \e^{\lambda_{i}y_{i}})$ as $n\to\infty$ for any $i=1,\ldots,s$.
Next, for $x\leq 0,\,y\geq 0$ we have
\begin{align*}
Q(x,y)&:=\Pro\left(\eta_j\leq x,\eta_k-\eta_j\leq y\mid \eta_j\text{ and $\eta_k$ are records}\right)\\
&=jk\Pro\left(\eta_j\leq x,\eta_k-\eta_j\leq y,\eta_j>\frac{\eta_1}{j-1},\eta_k>\max\left(\eta_j,\frac{\eta_2}{k-j-1}\right) \right)
\end{align*}
When $x\leq -y$,
\[
Q(x,y)=jk\int_{-\infty}^x\int_{z_j}^{z_j+y}\e^{(k-j)z_k}\e^{jz_j}\diff z_k\diff z_j=\frac{jk}{k-j}(\e^{(k-j)y}-1)\int_{-\infty}^x\e^{kz_j}\diff z_j=\frac{j}{k-j}(\e^{(k-j)y}-1)\e^{kx}.
\]
Therefore, $Q(x/n,y/n)\to \beta_1 \left(\e^{(\lambda_2-\lambda_1)y}-1\right)\e^{\lambda_2 x}$ as $n\to\infty$.
When $x>-y$
\[
\begin{split}
Q(x,y)&=jk\left(\int_{-\infty}^{-y}\int_{z_j}^{z_j+y}\e^{(k-j)z_k}\e^{jz_j}\diff z_k\diff z_j+\int_{-y}^x\int_{z_j}^0\e^{(k-j)z_k}\e^{jz_j}\diff z_k\diff z_j\right)\\
&=\frac{j}{k-j}(\e^{(k-j)y}-1)\e^{-ky}+\frac{jk}{k-j}\left(\frac{\e^{jx}-\e^{-jy}}{j}-\frac{\e^{kx}-\e^{-ky}}{k}\right)=\frac{k}{k-j}\e^{jx}-\frac{j}{k-j}\e^{kx}-\e^{-jy}.
\end{split}
\]
Therefore $Q(x/n,y/n)\to \beta_2 \e^{\lambda_1 x}-\beta_1\e^{\lambda_2 x}-\e^{-\lambda_1 y}$ as $n\to\infty$.
\end{proof}
\vskip-1cm
\section*{Acknowledgements}
This paper was written while the first author was a visiting professor at the Department of Decision Sciences, Bocconi University, Milan, Italy. He is grateful to his hosts for their hospitality and the extensively constructive atmosphere.
\bibliographystyle{chicago}
\bibliography{evt}
\end{document}